\documentclass[a4paper,12pt]{amsart}
\makeatletter
\@namedef{subjclassname@2020}{%
\textup{2020} Mathematics Subject Classification}
\makeatother
\usepackage[shortlabels]{enumitem}
\newtheorem{thm}{Theorem}[section]

\newtheorem{prop}[thm]{Proposition}
\newtheorem{proposition}[thm]{Proposition}
\newtheorem{cor}[thm]{Corollary}
\newtheorem{corollary}[thm]{Corollary}

\newtheorem{lemma}[thm]{Lemma}
\newcommand{\eps}{\varepsilon}

\renewcommand{\Im}{\operatorname{\rm{Im}}}
\renewcommand{\Re}{\operatorname{\rm{Re}}}

\def\C{\mathbb C}
\def\D{\mathbb D}
\def\N{\mathbb N}
\def\R{\mathbb R}

\def\ds{\displaystyle}
\def\k{\kappa}

\def\e{\varepsilon}
\def\d{\delta}

\def\z{\tilde z}
\def\w{\tilde w}

\newenvironment{proof*}{\vskip 2mm\noindent {}}{\hfill $\Box$ \vskip 2mm}

\begin{document}

\numberwithin{equation}{section}

\title[Quantitative localization and comparison]
{Quantitative localization and comparison of invariant distances of domains in $\mathbb C^n$}

\author[N. Nikolov]{Nikolai Nikolov}
\author[P.J. Thomas]{Pascal J. Thomas}

\address{N. Nikolov:
Institute of Mathematics and Informatics\\
Bulgarian Academy of Sciences\\
Acad. G. Bonchev Str., Block 8\\
1113 Sofia, Bulgaria\\ \smallskip\newline
Faculty of Information Sciences\\
State University of Library Studies
and Information Technologies\\
69A, Shipchenski prohod Str.\\
1574 Sofia, Bulgaria}
\email{nik@math.bas.bg}

\address{P.J.~Thomas:
Institut de Math\'ematiques de Toulouse; UMR5219 \\
Universit\'e de Toulouse; CNRS \\
UPS, F-31062 Toulouse Cedex 9, France} \email{pascal.thomas@math.univ-toulouse.fr}

\begin{abstract}
We obtain  explicit bounds on the difference
between ``local'' and ``global'' Kobayashi distances in a domain of $\mathbb C^n$
 as the points go toward a boundary point with appropriate
geometric properties. We use this for the global comparison
of various invariant distances.  We provide some sharp estimates
in dimension $1$.
\end{abstract}

\thanks{The first named author is partially supported by the National Science Fund,
Bulgaria under contract KP-06-N52/3.}


\maketitle

\section{Introduction}

\subsection{Previous results and motivations.}

In this paper, we want to study the behaviour of various invariant (pseudo)distances and metrics -- notably
the Kobayashi-Royden infinitesimal pseudometric -- as the points considered tend to the same
point on the Euclidean boundary of a domain in $\mathbb C^n$.  It is often useful or even necessary
to compare those metrics and distances to their counterparts defined
within a neighborhood of a boundary  point,
i.e. to get ``localization'' estimates of the kind already present in the seminal work of
Forstneri\v c and Rosay \cite{FR}.

For that purpose it will be necessary to see that, under some hypotheses of local
convexity, the whole of a geodesic (or
more generally of a curve close to the solution of the extremal problem)
has to tend to the boundary. Proposition
\ref{locgeod} makes this precise.  
This can be seen as a sort of converse to the ``visibility'' results; recall that
in a geodesic metric space, ``visibility of geodesics'' roughly means that when two points tend
to distinct boundary points (which may lie in some abstract boundary),
the geodesic curves connecting them have to pass through
a fixed compact set depending only on the limits.

Regarding the behavior of extremal curves, F. Bracci, J.E. Fornaess and E.F. Wold prove
that in a strongly pseudoconvex domain $D$, two points sufficiently close to the boundary
and whose difference is almost tangential to it can be joined by a
complex geodesic of $D$ which is also a holomorphic retract of $D$ (and remains close the points, see \cite[Proposition 2.5]{BFW}).
As a consequence, for those points, the Carath\'eodory distance coincides with the Lempert function (and
therefore with the Kobayashi distance).

Boundedness of differences between local and global Kobayashi distances (themselves tending to infinity)
was recently obtained under various hypotheses in \cite{LW} and \cite{BNT}. On the other hand,
several other invariant distances can be considered, as in \cite[Propositions 1 and 2]{N2}, which themselves make use of the precise
estimates obtained in \cite{BB} for strictly pseudoconvex domains.
A.  Zimmer \cite[Theorem 16.3]{Zim} obtained localization results for locally convexifiable
domains of finite type.
Much earlier, Venturini \cite[Theorem 1, Proposition 3]{Ven} considered the ratios of various invariant distances.
A motivation of the present paper is to obtain quantitative versions of the results
mentioned above from \cite{Ven}, \cite{N2} and \cite{Zim}.

Our main localization result is Theorem \ref{kobloc}, which gives an explicit bound on the difference
between local and global Kobayashi distances as the points go toward a boundary point with the required
geometric properties.  Theorem \ref{finsloc} is the general technical tool needed to obtain this localization result
from previous localization results on the infinitesimal level.

Another motivation is to obtain more precise estimates
about the Bergman metric and the Lempert function than those obtained in \cite{N2}.
The Bergman metric has a definition of a slightly different character, coming
from extremal problems on spaces of square-integrable holomorphic functions. But its values are  constrained
by the ``squeezing function'' (details and references given below), so we can deduce similar results about it,
see Theorem \ref{berg}, which is related to \cite[Proposition 2]{N2}.

Our second main result is Theorem \ref{lempcara} which provides  quantitative versions
of \cite[Theorem 1]{Ven} (for ratios of invariant distances, in the essential case
where  both points approach the same boundary point) and
of \cite[Proposition 1]{N2} (for differences). Note that
Theorems \ref{lemp} and \ref{lempcara} involve the Carath\'eodory distance and
the Lempert function, which are more delicate to manipulate since they are not inner distances.

Explicit calculations lead to a sharp localization result in dimension $1$, Theorem \ref{dimone}.
The question remains open of which is the correct exponent in the remainder term in higher dimension, at least for
strictly pseudoconvex domains.
\vskip.3cm
\noindent
{\bf Acknowledgement.} We would like to thank the referee for his very careful reading of the paper,
and his detailed suggestions which helped us correct some inaccuracies and improve the exposition.

\vskip.3cm

\subsection{Some definitions and notations.}

We denote by $\D$ the unit disc in $\C$.
Let $D$ be a domain in $\C^n.$ We denote by $\delta_D(z)$ the Euclidean distance from $z\in D$ to
$\C^n \setminus D$.

Let $x, y, z\in D$ and $X\in\C^n.$ The Kobayashi-Royden (pseudo)metric $\k_D$
and the Kobayashi (pseudo)distance $k_D$ of $D$ are defined as:
\begin{equation*}
\label{defgeod}
\k_D(z;X)=\inf\{|\alpha|:\exists\varphi\in\mathcal O(\D,D):
\varphi(0)=z,\alpha\varphi'(0)=X\},
\end{equation*}
\begin{equation}
\label{defkoba}
k_D(x,y)=\inf_\sigma\int_a^b\k_D(\sigma(t);\sigma'(t))dt,
\end{equation}
where the infimum is taken over all absolutely continuous curves $\sigma:[a,b]\to D$ with $a,b \in \R$,
$a<b$, $\sigma(a)=x$ and $\sigma(b)=y.$
\vskip.3cm
We will use the following notation: for any sets $D$ and $U$ in $\mathbb C^n,$ we write $D_U:=D\cap U.$
This will be used mostly for $D$ a domain in $\C^n$ and $U$ a neighborhood of $p \in \partial D$.

Recall that a domain $D$ is \emph{$\C$-convex} if any non-empty intersection of $D$ with a complex line
is connected and simply connected. For the relationship of this notion with (weak) lineal convexity, see
\cite{APS} or \cite{Hor}. Of course, any convex domain is $\C$-convex.

We say that a domain $D$ is \emph{convexifiable} (resp. $\C$-convexifiable) near a point $p \in \partial D$ if there exists
a neighborhood $U$ of $p$ and a biholomorphic map $\Phi$ defined in a neighborhood of $ \overline{D_U}$ such that
$\Phi(D_U)$ is convex (resp. $\C$-convex). If $\partial D$ is of class $\mathcal C^2$ near $p$,
a point of strict pseudoconvexity, then $D$ is
convexifiable near $p$.

We recall that when $D \subset \C^n$ is a domain with $\mathcal C^\infty$-smooth boundary near  a point $p\in \partial D$,
then $p$ is said to be of  \emph{type at most $m$} if any analytic curve  has an order of contact at most $m$ with
$\partial D$ at $p.$
When $D$ is $\C$-convex,
it is enough to consider the order of contact with affine complex lines
passing through $p$ (see \cite[Proposition 6]{NPZ}).
So when $D$ is $\C$-convex, we say that $p\in \partial D$ is of type at most $m$ when $\partial D$ is $\mathcal C^{m}$-smooth near $p$  ($m\in\N$) and the order of contact of complex lines with $\partial D$ is
at most $m$ in a neighborhood of $p$. If $D$ is $\C$-convexifiable near a point $p \in \partial D$,
and $\partial D$ is $\mathcal C^{m}$-smooth near $p$, we say that $p$ is of type at most $m$
if $\Phi(p)$ is with respect to $\Phi(D_U)$.

\subsection{Results.}

\begin{thm}
\label{kobloc}
Let $D \subset \C^n$ be a domain. Assume that $D$ is $\C$-convexifiable near $p \in \partial D$,
and that $p$ is of type at most $m.$
Then there exists a neighborhood $U_0$ of $p$ such that for any neighborhoods
$V\subset \subset U\subset U_0$, with $D_U$ connected,
 one may find $C>0$  such that
\begin{equation}
\label{kobloceq}
k_D(z,w) \le k_{D_{U}}(z,w) \le k_D(z,w) +C|z-w|^{1/m},\quad z,w\in D_V, \qquad \mbox{and}
\end{equation}
\begin{equation}
\label{kobrat}
1\le  \frac{k_{D_U}(z,w)}{k_D(z,w)} \le 1 +C'( \delta_D(z)+|z-w|^{1/m}),\quad z\neq w\in D_V.
\end{equation}
\end{thm}

This theorem follows quickly from the more general and technical Theorem \ref{finsloc}, see subsection \ref{pfcor}.

\vskip.3cm
\noindent
{\bf Remark.} In particular, if $p \in \partial D$ is a strictly pseudoconvex point,
then local convexifiability follows and the conclusions \eqref{kobloceq} and \eqref{kobrat}
hold with $m=2$.

The statement \eqref{kobrat} for $m=2$ provides a quantitative refinement of the results
about the behavior of the ratios of various distances
 in \cite[Proposition 3]{Ven} and \cite[Proposition 4(5)]{N2}.
The statement \eqref{kobloceq} somewhat generalizes and sharpens  \cite[Theorem 16.3]{Zim}.
\vskip.3cm

We need a couple of definitions to state the next result.
Let us say that $f$ is an \emph{admissible weight} if $f:(0,\infty)\to(0,\infty)$ is a
continuous, increasing function
such that $x\mapsto \frac{f(x)}x$ is decreasing and $\int_0^1 \frac{f(x)}x dx < \infty$.
Without loss of generality, we may assume that $\lim_{x\to\infty} f(x)=\infty$.

Set  $\omega_f(s):=\int_0^s \frac{f(x)}x dx.$

We let $t_D:D\times\C^n\to[0,\infty]$ be an upper semicontinuous Finsler pseudometric, and always denote by
$T_D$  the pseudodistance obtained from integrating $t_D.$

\begin{thm}
\label{finsloc}

Let $D \subset \C^n$ be a domain. Assume that $D$ is $\C$-convexifiable near $p \in \partial D$,
and that $p$ is of type at most $m.$


Then there exists $U_0$ a neighborhood of $p$ with the following properties.

For any bounded neighborhood $U$ of $p$ such that $U\subset U_0$ and $D_U$ is connected,
 any $f$ admissible weight and any $t_D$ (u.s.c.) Finsler pseudometric such that
\begin{equation}
\label{metricupest}
t_D (z;X)\le\left( 1+ f\left(\delta_D(z)\right) \right) \kappa_D (z;X), z \in D_U,
\end{equation}
then for any neighborhood
$V\subset \subset U$ of $p$ one may find $C>0$  such that
\begin{equation}
\label{distupest}
T_D (z;w) \le k_D (z;w) + C \omega_f( |z-w|^{1/m}) ,\quad z,w\in D_V, 
\qquad \mbox{and}
\end{equation}
\begin{equation}
\label{uprat}
 \frac{T_D(z,w)}{k_D(z,w)} \le 1+ C f\left(\delta_D(z)+ |z-w|^{1/m}\right),\quad z\neq w\in D_V.
\end{equation}

Furthermore, if instead of \eqref{metricupest}, we assume
\begin{equation}
\label{metriclowest}
t_D (z;X) \ge \left( 1+ f(\delta_D(z)) \right)^{-1} \kappa_D (z;X),  z \in D_U,
\end{equation}
then for any neighborhood
$V\subset \subset U$ of $p$ one may find $C>0$  such that
\begin{equation}
\label{distlowest}
T_D (z;w) \ge k_D (z;w) - C  \omega_f( |z-w|^{1/m}),\quad z,w\in D_V, \qquad \mbox{and}
\end{equation}
\begin{equation}
\label{lowrat}
 \frac{k_D(z,w)}{T_D(z,w)} \le 1+ C f\left( \delta_D(z)+ |z-w|^{1/m}\right),\quad z\neq w\in D_V.
\end{equation}
\end{thm}

We provide the proof of Theorem \ref{finsloc} in Section \ref{pf2}.

\vskip.3cm

\noindent
{\bf Remark.} In this theorem, and in all the results that follow, the
expression $|z-w|^{1/m}$ in the bounds
could be replaced by a somewhat better one, $B_D(z,w)$, defined by
\begin{equation}
\label{betterb}
B_D(z,w) :=  \left( \frac{|z-w|}{|z-w|^{1/2}+ \delta_D(z)^{1/2} +\delta_D(w)^{1/2} }\right)^{2/m}.
\end{equation}
We indicate how to modify the proof of Proposition  \ref{locgeod} at the end of Subsection \ref{notfar};
the reader can verify that Lemma \ref{expdec}, and subsequent steps of the proof, with suitably modified
statements, can be proved in the same way.

\vskip.3cm

For a bounded domain $D \subset \C^n$, we denote by $L_h^2(D)$ the space of holomorphic
functions which are square-integrable with respect to the usual Lebesgue measure.

We define the Bergman metric by
$$
\beta_D(z;X)=\frac{M_D(z;X)}{K_D(z)}, \quad z\in D,\  X\in\C^n,
$$
where
$$
M_D(z;X)=\sup\{|f'(z)X|:f\in L_h^2(D),\;||f||_D\le 1,\;f(z)=0\},
$$
and
$$
K_D(z)=\sup\{|f(z)|:f\in L_h^2(D),\;||f||_D\le 1\}
$$
is the square root of the Bergman kernel on the diagonal.

Furthermore, it will be convenient to have a different normalization of the Bergman metric
and to set
$$\tilde \beta_D(z;X) := \beta_D(z;X)/ \sqrt{n+1}.$$
We then define $b_D$ to be the inner distance obtained by integrating $\tilde \beta_D$.

\begin{thm}
\label{berg}
Let $D \subset \C^n$ be a bounded pseudoconvex domain. We denote by $u,v$ any two of the four pseudodistances
$ k_D, k_{D_U}, b_{D},b_{D_U}$.

Let $p\in \partial D$ be a strictly pseudoconvex point;
assume that $\partial D$ is $\mathcal C^{k,\eps}$-smooth in a neighborhood
of $p$, with $k= 2$ or $3$, and
$\eps\in (0,1]$ if $k=2$,  $\eps\in [0,1)$ if $k=3$.

Then there exists $U_0$
a neighborhood of $p$ such that for any open neighborhood of $p$, $U\subset U_0$,
with $D_U$ connected, and for any
neighborhood $V$ of $p$ such that $ V \subset \subset U$,
 there exists a constant $C>0$ such that for any $z\neq w \in D_V$
\begin{equation}
\label{kobberg}
| u(z,w)-v(z,w)| \le C|z-w|^{(k-2+\eps)/4}\mbox{, resp. }
\left| \frac{u(z,w)}{v(z,w)} -1\right| \le C\left( \delta_D(z)+|z-w|^{1/2}\right)^{(k-2+\eps)/2}.
\end{equation}

\end{thm}
The proof is given in Section \ref{lempert}.


\begin{cor}
\label{cor14}
Let $D$ be a strictly pseudoconvex domain
with $\mathcal C^{k,\eps}$-smooth boundary, with $k= 2$ or $3$, and
$\eps\in (0,1]$. Then there is a constant $C>0$ such that for any $z\neq w \in D$,
\begin{multline*}
| k_D(z,w)-b_D(z,w)| \le C|z-w|^{(k-2+\eps)/4},
\\
\left| \frac{k_D(z,w)}{b_D(z,w)} -1\right| \le C\left( \delta_D(z)+|z-w|^{1/2}\right)^{(k-2+\eps)/2}.
\end{multline*}
\end{cor}

A non-quantitative version of this statement for the $\mathcal C^2$-smooth case was given in \cite[Proposition 4(6)]{N2}.
\vskip.3cm

To compare various invariant distances, we will use the following (unbounded) version of the Lempert function, i.e.
\[
l_D(z,w):=\inf\{ \tanh^{-1}\alpha : \alpha\in[0,1) \mbox{ and }\exists\varphi\in\mathcal O(\mathbb
D,D):\varphi(0)=z,\varphi(\alpha)=w\}.
\]
Note that $k_D$ is the largest pseudodistance on $D$ which does not exceed $l_D$.

Similarly, let also $c_D$ be the (unbounded) Carath\'eodory distance,
\[
c_D(z,w):=\sup\{ \tanh^{-1}\alpha : \alpha\in[0,1) \mbox{ and }\exists f\in\mathcal O(D,\mathbb
D):f(z)=0,f(w)=\alpha\}.
\]

We always have $c_D \le k_D\le l_D$. We know that $l_D-c_D$ is bounded
on strictly pseudoconvex domains \cite[Proposition 1]{N2}.
We can refine this estimate in two directions: first, replacing the constant bound by a
quantity that depends on $|z-w|$, and second, for $l_D-k_D$, to have a result covering the case
of domains of type $m$.

\begin{thm}
\label{lemp}
Let $D \subset \C^n$ be a bounded domain with $\mathcal C^{m}$-smooth boundary.
Assume that $D$ is $\C$-convexifiable near each $p \in \partial D$,
and that any $p \in \partial D$ is of type at most $m.$

Then there exists $C>0$ such that
\begin{equation}
\label{lkvd}
0\le l_D(z,w) - k_D(z,w) \le C |z-w|^{1/m},\quad z,w\in D,  \qquad \mbox{and}
\end{equation}
\begin{equation}
\label{lkrat}
1\le \frac{l_D(z,w)}{k_D(z,w)} \le 1+ C \left( \delta_D(z)+|z-w|^{1/m}\right),\quad z\neq w\in D.
\end{equation}
\end{thm}

Comparing with the Carath\'eodory distance seems to require strict pseudoconvexity.

\begin{thm}
\label{lempcara}
If $D \subset \C^n$ is a strictly pseudoconvex domain, then
there exists $C>0$ such that
\begin{equation}
\label{lcvd}
0\le l_D(z,w) - c_D(z,w) \le C |z-w|^{1/2},\quad z,w\in D  \qquad \mbox{and}
\end{equation}
\begin{equation}
\label{lcrat}
1\le \frac{l_D(z,w)}{c_D(z,w)} \le 1+ C \left( \delta_D(z)+|z-w|^{1/2}\right),\quad z\neq w\in D.
\end{equation}
\end{thm}
The two theorems above are proved in Section \ref{lempert}.

In the special case of dimension $1$, a much more precise estimate holds.

\begin{thm}
\label{dimone}
Let $p$ be a Dini-smooth boundary point
of a planar domain $D.$ Then there exists a neighborhood $U_0$ of $p$
such that for any neighborhoods $V\subset\subset U\subset U_0$ of $p$
one may find $C>0$ such that
$$0\le k_{D_U}(z,w)-c_D(z,w)\le C|z-w|(|z-w|+\d_D(z)^{1/2}\d_D(w)^{1/2}),\quad z,w\in D_V.$$
\end{thm}
The proof is to be found in Section \ref{plane}, as well as
examples showing that none of the terms in the parenthesis can be dispensed with.

\noindent{\bf Remarks.} a) Note that
$$\max\{k_D-c_D,c_{D_U}-c_D,k_{D_U}-k_D\}\le k_{D_U}-c_D,$$
so the three differences inside the maximum satisfy the estimate in Theorem \ref{dimone}.

b) Note that $k_{D_U}(z,w)-c_D(z,w)=|z-w|o(1)$ near $p$ but the exponent $1$
of $|z-w|$ cannot be improved. So $1$ is the right exponent in the plane.
We may conjecture the same in the strictly pseudoconvex case in $\mathbb C^n$
(our results only give 1/2).

c) Since $|z-w| \ge |\delta_D(z)- \delta_D(w)|$, if we allow a multiplicative constant, the second
term in the parenthesis could be $\max (\delta_D(z), \delta_D(w))$ or $(\delta_D(z)\delta_D(w))^{1/2}$
or in fact any number between $\delta_D(z)$ and $\delta_D(w)$.  This is true for all $n\ge 1$.

\section{Proofs of Theorem \ref{kobloc} and \ref{finsloc} }
\label{pf2}

We need to specify the class of curves that we will use instead of actual geodesics.
Let $t_D$ be a (u.s.c.) Finsler pseudometric and $T_D$ the associated integrated pseudodistance.

For any $\eps\ge 0$, we will call an absolutely continuous curve
$\sigma: [a;b] \rightarrow D$, 
with $\sigma(a)=z$, $\sigma(b)=w$,
 \emph{$\eps$-extremal} for $T_D$ if $\int_a^b t_D(\sigma(t);\sigma'(t))dt \le T_D(z,w) +\eps$.

 When $\sigma$ is an $\eps$-extremal curve for $T_D$, the definition of the pseudodistance and the triangle inequality
imply that for $a\le t_1<t_2\le b$,
\begin{equation}
\label{partlength}
T_D(\sigma(t_1), \sigma(t_2)) \le \int_{t_1}^{t_2} t_D(\sigma(t);\sigma'(t))dt \le T_D(\sigma(t_1), \sigma(t_2)) +\eps.
\end{equation}

\subsection{Proof of Theorem \ref{kobloc}.}
\label{pfcor}

\begin{proof}
The left hand side inequality follows from the definitions.

To prove the right hand side, choose $U_0$ as in Theorem \ref{finsloc}. Let $V\subset \subset U\subset U_0$,
with $D\cap U $ connected.

Consider the Finsler metric $t_D$
on $D$ given by $t_D(z;X):= \kappa_{D_U}(z;X)$ for $z\in D_U$ and $t_D(z;X):= \infty$ for $z\in D\setminus U$.
Then clearly $T_D|_{D_U}=k_{D_U}$. We now need to show that the hypotheses of Theorem \ref{finsloc} hold
on an appropriate open set.

\begin{lemma}
\label{connhd}
There exists an open set $W_0$ such that $ V\subset \subset W_0 \subset \subset U$ and $D\cap W_0$
is connected.
\end{lemma}

Accepting the Lemma temporarily,
for each point $q\in U \cap \overline D$, \cite[Corollary 6.14]{BNT} implies that
there exists a neighborhood $V_q \subset \subset U$ such that
\begin{equation}
\label{kobupest}
\k_{D_U}(z;X) \le (1+c\delta_D(z))\k_D(z;X),\quad z\in D \cap {V_q},\ X\in\C^n,
\end{equation}
and shrinking $V_q$ as needed, we can take it so that $D \cap {V_q}$ is connected and
 $f$ is of the form $f(s)= c_q s$ on $V_q$.

Take a finite covering of $\overline W_0 \cap \overline D$ by $V_{q_j}$, $q_j\in \overline W_0$, $1\le j \le N$,
and let $W_1:= \bigcup_{j=1}^N V_{q_j} \subset \subset U$.
Since $D\cap W_1  = (D\cap W_0) \cup \bigcup_{j=1}^N (D\cap V_{q_j})$, it is connected as well.
On $W_1$, we have \eqref{kobupest} with $f(s) = c s$, $c:=\max_{1\le j \le N} c_{q_j}$.

We have $V \subset \subset W_1$, and we can apply
Theorem \ref{finsloc}  with $U:=W_1$ and $f(s)=cs$; then $\omega_f(s)=c's$ and
 \eqref{distupest}  yields the result.

To get \eqref{kobrat}, do as above but use \eqref{uprat} in Theorem \ref{finsloc}.
\end{proof}

\begin{proof*}
{\it Proof of Lemma \ref{connhd}.}


{\bf Claim.} For any $\eta>0$, there exists $\eps>0$ such that for any $z', z''\in  D\cap U$
with $\delta_U(z'), \delta_U(z'') \ge \eta$, then there exists a continuous curve $\gamma$ in $ D\cap U$
joining $z'$ to $z''$ such that for any $t$, $\delta_U(\gamma(t)) \ge \eps$.

Indeed, suppose that the Claim fails, and let $(z'_k, z''_k)_{k\ge 1}$ be such that
$$
\sup \left\{ \inf_{0\le t\le 1}\delta_U(\gamma(t)): \gamma(0)= z'_k, \gamma(1)= z''_k\right\}
\to 0 \mbox{ as } k \to \infty.
$$
Pick convergent subsequences tending to $z'_\infty, z''_\infty$ respectively.  If both limits
belong to $D_U$, some continuous curve connects them within $D_U$ and it violates the property above.
If one of the points, say $z'_\infty$, is in $\partial D$, then using the fact that $\partial D$ is
$\mathcal C^1$-smooth, we can find $r\in (0,\delta_U(z'_\infty))$ such that
$B(z'_\infty,r) \cap D \ni z'_\infty + t n_{z'_\infty}$
for any $t\in (0,r)$, where $n_p$ stands for the inner normal at a point $p\in\partial D$ (as usual,
we write $B(a,r)$ for the Euclidean ball centered at $a$, of radius $r$).
Then, using again the fact that $\partial D$ is $\mathcal C^1$-smooth, for $k$ large enough, the line segment from
$z'_k$ to $z'_\infty + \frac{r}2 n_{z'_\infty}$ is contained in $\bar B(z'_\infty,r) \cap D \cap U$
 and remains at distance at least $r/2$ from $\C^n\setminus U$; then, as before,
$z'_\infty +\frac{r}2 n_{z'_\infty}$ can be connected to $z''_\infty$ by
a curve which stays compactly within $D\cap U$,
and this violates our assumption again. If both $z'_\infty, z''_\infty \in \partial D$,
we construct two balls and two inward-pointing line segments as above, and connect their inner
extremities, again obtaining a contradiction.

Now let $\eta:=\mbox{dist}(V, \C^n\setminus U)$, by hypothesis $\eta>0$; choose $\eps>0$
depending on $\eta$ as in the Claim. For $z,w \in V \cap \overline D$, pick a curve $\gamma$ connecting them
with $\delta_U(\gamma(t)) \ge \eps$, and let
\[
N(z,w):= \bigcup_{0\le t\le 1} B\left(\gamma(t), \frac12 \min[\delta_U(\gamma(t)), \delta_D(\gamma(t))]  \right).
\]
That set is connected and open, $N(z,w)\subset D\cap U$ and $\mbox{dist}(N(z,w), \C^n\setminus U) \ge \eps/2$.
We take $W'_0:=  \bigcup_{z,w \in  D \cap V} N(z,w)$, which is connected and relatively compact in $U$;
 $W_0=V\cup W'_0 = (V\setminus D) \cup W'_0$.
Then $D\cap V \subset W'_0 = D\cap W_0$, and $W_0 \subset \subset U$.
\end{proof*}

\subsection{Kobayashi $\eps$-extremals cannot go far}
\label{notfar}

Theorem \ref{finsloc} will be proved with the help of several auxiliary results, some of which will be
ultimately superseded but are needed as intermediate steps.

\begin{prop}
\label{locgeod}
Let $D \subset \C^n$ be a   domain, $p\in \partial D$ such that  $D$ is $\C$-convexifiable near $p$. Let $U$
be a neighborhood of $p$ such that $D_U$  is biholomorphic to a $\C$-convex open set
under a biholomorphism of a neighborhood of $ \overline{D_U}$, and such that all points of
$U\cap \partial D$ are of type at most $m$.

Then  there exist a constant $C>0$ and a neighborhood $V \subset U$ of $p$ such that for any $z\neq w \in V$,
there exists $\eps_{z,w}>0$ such that
for any $\eps\in (0,\eps_{z,w})$,
for any 
$\eps$-extremal curve $\sigma$ for $k_D$ joining $z$ to $w$,
\begin{equation}
\left| \sigma(s)-z \right| \le C |z-w|^{1/m}, \mbox{ for any } s \in [a,b].
\end{equation}
An immediate consequence is that there exists a neighborhood $V'\subset V$ of $p$ such that
for any $z,w \in V'$, there exists  $\eps_{z,w}>0$
such that for any $\eps \in (0,\eps_{z,w})$, any 
$\eps$-extremal curve $\sigma$ for $k_D$ joining $z$ to $w$
satisfies $\sigma(s) \in D_U$, $a\le s \le b$.
\end{prop}

Note that we do not need to assume that $D$ itself is bounded, or even hyperbolic, in the above result.
All the action will take place at a local level, where we can reduce ourselves to the case of bounded
$\C$-convex domains. Note also that by \eqref{kobupest}, we have $\kappa_D (z;X) \ge c \|X\|$ on
each $V_q$, once $U_0$ has been chosen small enough.

\begin{proof}
We may assume that $V \subset B(p,\frac12)$, and $C\ge 2$.  Then if $q \in D \setminus B(z, C |z-w|^{1/m})$,
$|q-w|\ge (C-1) |z-w|^{1/m}$.

Shrinking $V$ further, we know by \cite[Theorem 7]{NA} that, for $z,w \in V$,
\begin{equation}
\label{NAupest}
k_D(z,w) \le \log \left( 1+ \frac{2 |z-w|}{\delta_D(w)^{1/2}\delta_D(z)^{1/2}} \right).
\end{equation}

Let $\sigma$ be an $\eps$-extremal curve.

Suppose that there exists $s_0 \in [a;b]$ such that $|\sigma(s_0) -z| > C |z-w|^{1/m}$.
Let $s_1:=\sup \{s: \sigma([a;s]) \subset B(z, C |z-w|^{1/m}) \}$,
$s_2:=\inf \{s: \sigma([s;b]) \subset B(z, C |z-w|^{1/m}) \}$, then $a<s_1<s_0<s_2<b$ and if we let $z':=\sigma(s_1)$,
$w':= \sigma(s_2)$, we have $z',w' \in \partial B(z, C |z-w|^{1/m})$.

Then
\begin{multline*}
k_D(z,w) \ge \int_{a}^{b} \k_D(\sigma(t);\sigma'(t))dt  -\eps  \\
\ge   \int_{a}^{s_1} \k_D(\sigma(t);\sigma'(t))dt +
\int_{s_2}^{b} \k_D(\sigma(t);\sigma'(t))dt - \eps \ge k_D(w,w')+k_D(z',z)- \eps.
\end{multline*}
We may reduce $V$ further so that $B(z, C |z-w|^{1/m})$ is contained in a small enough neighborhood of $p$
to apply to it \cite[Theorem 1.7]{NTh}, which yields a constant $C_1>0$ such that
\begin{equation}
\label{pisalow}
k_D(w,w') \ge \frac{m}2 \log \left( 1+C_1 \frac{|w'-w|}{\delta_D(w)^{1/m}} \right) \mbox{ and }
k_D(z,z') \ge \frac{m}2 \log \left( 1+C_1 \frac{|z'-z|}{\delta_D(z)^{1/m}} \right) ,
\end{equation}
so
\[
k_D(z,w) \ge \frac{m}2 \log \left( 1+C_1^2 \frac{C(C-1) |z-w|^{2/m}}{\delta_D(w)^{1/m}\delta_D(z)^{1/m}} \right) - \eps
\ge \log \left( 1+C_1^{m} \frac{(C-1)^m |z-w|}{\delta_D(w)^{1/2}\delta_D(z)^{1/2}} \right)- \eps.
\]
Choosing $\eps$ small enough, depending on $z$ and $w$, we have
\[
k_D(z,w) \ge \log \left( 1+C_1^{m} \frac{(C-1)^m |z-w|}{2\delta_D(w)^{1/2}\delta_D(z)^{1/2}} \right).
\]
Choosing $C$ large enough so that $C_1^{m}(C-1)^m > 4$, we reach a contradiction with \eqref{NAupest}.
\end{proof}

\noindent
{\bf Remark.}
Let $B_D(z,w)$ be as in \eqref{betterb}. We may improve Proposition \ref{locgeod} to obtain
\begin{equation}
\label{better21}
\left| \sigma(s)-z \right| \le B_D(z,w), \mbox{ for any } s \in [a,b].
\end{equation}
In fact, it will be enough to show
\begin{equation}
\label{overdelta}
\left| \sigma(s)-z \right| \le C \delta_D(z)^{-1/m}|z-w|^{2/m};
\end{equation}
we can then exchange the roles
of $z$ and $w$, and finally thanks to Proposition \ref{locgeod}, get
\begin{equation*}
\left| \sigma(s)-z \right| \le C \min \left( |z-w|^{1/m}, \frac{|z-w|^{2/m}}{ \delta_D(z)^{1/m}} , \frac{|z-w|^{2/m}}{ \delta_D(w)^{1/m} }\right),
\end{equation*}
which is equivalent to \eqref{better21}.

To prove \eqref{overdelta}, the significant case is when $|z-w| \le \frac12 \delta_D(z)$, and so $\delta_D(z) \asymp \delta_D(w)$.
Assume that the $\eps$-extremal curve gets out of $B(z,C \delta_D(z)^{-1/m}|z-w|^{2/m} )$ for a large $C$;
we define $z',w'$ as before, we have $k_D(z,w) \ge k_D(z,z')-\eps$,
so from \eqref{pisalow} we deduce this time
\[
k_D(z,w) \ge \frac{m}2 \log \left( 1+C_1 \frac{C |z-w|^{2/m}}{\delta_D(z)^{1/m} \delta_D(z)^{1/m}} \right) - \eps
\ge \log \left( 1+C_2 \frac{C |z-w|}{\delta_D(z)} \right),
\]
for $\eps$ small enough.  As before we get a contradiction with \eqref{NAupest} when $C$ is too big.

\subsection{Proof of \eqref{distupest}.}

\label{pfup}

We will need the following technical result.

\begin{lemma}
\label{locCcvx}
Let $D$ be a domain with boundary $\partial D$ of class $\mathcal C^k$, $k\ge 2$. Suppose that
$D$ is $\C$-convexifiable near a point $p \in \partial D$.  Then for any small enough neighborhood $W_1$ of $p$,
with $W_1$ and $D\cap W_1$ connected,
there exists a neighborhood $W_2 \subset W_1$ of $p$ such that $D\cap W_2$ is a $\mathcal C^k$-smooth
$\C$-convexifiable domain.
\end{lemma}
\begin{proof}
By the hypothesis, there is a neighborhood $W_0$ of $p$ and $\Phi$ a biholomorphism
on a neighborhood $V$ of $\overline{D\cap W_0}$ such that $\Phi(D\cap W_0)$ is $\C$-convex.

Let $W_1\subset W_0$ be a neighborhood of $p$ such that $D\cap W_1$ is connected.  Then
$\Phi(D\cap W_1)$ satisfies the hypotheses of \cite[Proposition 3.3]{NPT2} for the points
of $\Phi(\partial D)$ near $\Phi(p)$. So we obtain a neighborhood $U$ of $\Phi(p)$, which we
may reduce to have $U\subset \Phi(V\cap W_1)$, and $G_1$ a $\mathcal C^k$-smooth domain such that
$G_1\subset \Phi(D\cap W_1)$ and $\Phi(D\cap W_1) \cap U = G_1 \cap U$.

We then take $G:= \Phi^{-1} (G_1) \subset D\cap W_1$, $W_3:= \Phi^{-1} (U) \subset  W_1$.
We have $D\cap W_3= D\cap W_1\cap W_3 =G\cap W_3$. Now take $W_2:= G \cup W_3$.
\end{proof}

We first prove a local version of the result.

\begin{lemma}
\label{locprelim}
Under the hypothesis \eqref{metricupest} of Theorem \ref{finsloc}, there exists a neighborhood $U_0$ of $p$ such that
for any point $p'\in U_0 \cap \partial D$, there exists a neighborhood $V_{p'}$ of $p'$ such that
\eqref{distupest} holds for any $z,w \in D\cap V_{p'}$.
\end{lemma}

\begin{proof}
We start with a neighborhood $U_0$ such that every point of $U_0\cap \partial D$ is of type at most $m$, and
$\C$-convexifiable (those are open properties and $p$ satisfies them).
By Lemma \ref{locCcvx}, we may reduce $U_0$ and assume that $D_{U_0}$ is Dini-smooth.

By \cite[Lemma 6.10]{BNT}, proved in \cite[Proposition 4.4]{BZ},
for any $\eps>0$, there exists an $\eps$-extremal curve for the Kobayashi metric
joining $z$ to $w$, $\sigma : [a;b] \rightarrow D$, which
is absolutely continuous and which
we may parametrize by  Kobayashi-Royden length, i.e. $\kappa_D(\sigma(s);\sigma'(s))=1$ a.e. in $s \in [a;b]$.
By Proposition \ref{locgeod}, $\sigma([a;b]) \subset D_{U_0}$ if $V_{p'}$ and $\eps$ are small enough.
Then
\begin{multline}
\label{integbd}
T_D(z,w) \le \int_a^b t_D ( \sigma(t);\sigma'(t)) dt
\\
\le
\int_a^b \left( 1+ f\left(\delta_D(\sigma(t))\right) \right) dt
\le k_D(z,w) + \eps + \int_a^b  f\left(\delta_D(\sigma(t))\right) dt.
\end{multline}
Now we need to see that when $b$ is large, the last integral remains bounded.
We will achieve this by proving that as $t$ tends to infinity, $\sigma(t)$ must tend to the boundary
at a certain rate.

The proof below is inspired by \cite[Theorem 6.5]{BNT}
and \cite{LW}; see also a similar inequality in \cite[Proof of Theorem 1.4, Claim 2]{BZ}.

\begin{lemma}
\label{expdec}
There exists $s_0 \in [a;b]$ and $C_1, C_2>0$, independent of $a,b$ and $\eps$ provided it is
small enough, such that for any $t\in [a;b]$
with $|t-s_0| \ge C_2$,
\[
\delta_D(\sigma(t)) \le C_1 |z-w|^{1/m} e^{-|t-s_0|}.
\]
\end{lemma}
\begin{proof}
For $z,w \in D\cap V_{p'}$, $V_{p'}$
a small enough neighborhood of $p'$, $\delta_{D_{U_0}}(\sigma(t)) = \delta_D(\sigma(t))$ for
any $t\in [a;b]$, by Proposition \ref{locgeod}.
For any $t,s \in [a;b]$, $\sigma (s), \sigma (t) \in D_{U_0}$,
and since $D_{U_0}$ is Dini-smooth we can apply the upper
estimate for the Kobayashi distance from \cite[Corollary 8]{NA} to it:
\begin{equation}
\label{logup}
-\eps + |t-s| \le  k_D ( \sigma (s), \sigma (t) ) \le k_{D_{U_0}} ( \sigma (s), \sigma (t) )
\le
\log \left( 1+ \frac{C_3 |\sigma (s) - \sigma (t)|}{\delta_D(\sigma (s))^{1/2}\delta_D(\sigma (t))^{1/2}} \right).
\end{equation}
This implies after exponentiation
\[
\delta_D(\sigma (s))^{1/2}\delta_D(\sigma (t))^{1/2} \le
\frac{C_3 |\sigma (s) - \sigma (t)|}{e^{|t-s|-\eps}-1} \le \frac{C_4 |z-w|^{1/m}}{e^{|t-s|-\eps}-1},
\]
by Proposition  \ref{locgeod}. Now choose $s=s_0 \in [a;b]$ such that $\delta_D(\sigma (s_0))$ is maximal, so that
the left hand side above is always larger than $\delta_D(\sigma (t))$.

Outside of an interval of fixed
length around $s_0$, $e^{|t-s_0|-\eps}-1 \ge \frac12 e^{|t-s_0|}$,
and therefore $\delta_D(\sigma (t)) \le C_1 |z-w|^{1/m}  e^{-|t-s_0|}$.
This finishes the proof of Lemma \ref{expdec}.
\end{proof}

Now from \eqref{integbd} we deduce
\begin{equation}
\label{TkCint}
T_D(z,w) \le k_D(z,w) + \eps+ 2 \min(C_2, b-a) f(\delta_D(\sigma (s_0))) + 2 \int_0^\infty f\left(C_1 |z-w|^{1/m} e^{-t}\right) dt,
\end{equation}
and making the change of variables $x:= C_1 |z-w|^{1/m} e^{-t}$ we see that the last integral is bounded by
$\omega_f \left( C_1 |z-w|^{1/m} \right)$.

By Proposition \ref{locgeod}, $\delta_D(\sigma (s_0)) \le \delta_D(z)+C_5|z-w|^{1/m}$.
We have two cases.

{\it Case 1:} $\delta_D(z)\le 2 C_5|z-w|^{1/m} $.

Then $f(\delta_D(\sigma (s_0))) \le f((C_5+2) |z-w|^{1/m}) \le \omega_f \left((C_5+2) |z-w|^{1/m}\right)$,
and we get \eqref{distupest} by letting $\eps$ tend to zero and
setting $C:= \max(C_1, C_5+2, 2+ 2 C_2)^2 $ and noticing that $\omega_f(Cx)\le C\omega_f(x)$.

{\it Case 2:} $\delta_D(z)\ge 2 C_5|z-w|^{1/m} $.

The hypothesis of Case 2 implies $\delta_D(w)\ge \frac12\delta_D(z)$ and $\delta_D(\sigma (s_0)) \le \frac32 \delta_D(z)$.
Applying \eqref{NAupest}, we see that for $\eps$ small enough,
\[
b-a \le 2k_D(z,w) \le 2 \frac{2 |z-w|}{\delta_D(w)^{1/2}\delta_D(z)^{1/2}}  \le \frac{4\sqrt2 |z-w|}{\delta_D(z)} .
\]
So
\[
 (b-a) f(\delta_D(\sigma (s_0))) \le C' |z-w| \frac{f(\delta_D(\sigma (s_0)))}{\delta_D(\sigma (s_0))}
 \le C' |z-w|\frac{f( |z-w|^{1/m})}{ |z-w|^{1/m}} \le C'' f( |z-w|^{1/m}),
 \]
 and we conclude as in Case 1.
This finishes the proof of Lemma \ref{locprelim}.
\end{proof}

\begin{proof*}{\it End of the proof of Theorem \ref{finsloc}, \eqref{distupest}.}

We proceed by contradiction and assume that there are two sequences $(z_k), (w_k)$ in $V$ such that
\[
T_D (z_k;w_k) - k_D (z_k;w_k) \ge k \omega_f (|z_k-w_k|^{1/m}).
\]
 Passing to subsequences, we may assume
that $z_k \to z_\infty \in \overline V$ and $w_k \to w_\infty \in \overline V$.

{\it Case 1: $z_\infty=w_\infty$.}

If $z_\infty \in \partial D$, then 
for $k$ large enough, $z_k, w_k \in V_{z_\infty}$, and Lemma \ref{locprelim} contradicts the assumption.

If $z_\infty \in  D $, then for $k$ large enough, $z_k, w_k$ are away from $\partial D$,
and   $\delta_D (\sigma (t))$ is bounded from above and below for all $t\in [a;b]$. Then $T_D (z_k;w_k) - k_D (z_k;w_k)
\le C |z_k- w_k|$, which contradicts the assumption again because $\lim_{x\to 0^+} \frac{ f(x)}x >0$.

{\it Case 2: $z_\infty \neq w_\infty$.}

Here the hypotheses imply in particular
\[
\lim_{k\to\infty} k \omega_f (|z_k-w_k|^{1/m}) = \infty.
\]

{\it Case 2.1:} $z_\infty, w_\infty \in D$.  Since
$D_V$ is connected, $T_D(z_k,w_k)$ is bounded above, which contradicts the assumption.

{\it Case 2.2:} $z_\infty, w_\infty \in \partial D$.
We can pick neighborhoods $\tilde V_z \subset V_{z_\infty}$, $\tilde V_w \subset V_{w_\infty}$, such that
$\overline{\tilde V_z} \cap \overline{\tilde V_w}=\emptyset$. Let $N_p$ denote
the inner normal half-line at a boundary point $p$, and
pick points $z' \in N_{z_\infty} \cap \partial \tilde V_z$, $w' \in N_{w_\infty} \cap \partial \tilde V_w$.
Reducing $\tilde V_z, \tilde V_w$ if needed, we may assume
$ \max(|z_\infty-z'|, |w_\infty-w'|) < |z_\infty- w_\infty|$.

For $k$ large enough, by Lemma \ref{locprelim} and the above inequality,
\begin{multline}
\label{tbeg}
T_D(z_k,w_k) \le T_D(z_k,z') +  T_D(z',w')+ T_D(w',w_k) \\
\le
T_D(z',w')+ k_D(z_k,z') + k_D(w',w_k) + 2 C \omega_f (|z_\infty- w_\infty|).
\end{multline}
By \cite[Theorem 1.6]{NTh},   shrinking $U_0$ if needed (using Lemma \ref{locCcvx}),
\begin{multline*}
k_D(z_k,w_k) \ge  \frac{m}2 \log\left(  1+ \frac{C |z_k-w_k|}{\delta_D(z_k)^{1/m}}\right)
\left(  1+ \frac{C |z_k-w_k|}{\delta_D(w_k)^{1/m}}\right)
\\
\ge
\log \left(   \frac{ |z_k-w_k|^{m/2}}{\delta_D(z_k)^{1/2}}\right) +
\log \left(   \frac{ |z_k-w_k|^{m/2}}{\delta_D(w_k)^{1/2}}\right)-C
\ge \log    \frac{ 1}{\delta_D(z_k)^{1/2}}+\log    \frac{ 1}{\delta_D(w_k)^{1/2}} -C,
\end{multline*}
where $C$ is some constant depending on $z_\infty, w_\infty$;
while by the introduction of \cite{NPT1}, itself relying on \cite{FR} and \cite{JP},
\[
k_D(z_k,z') \le \frac12 \log \left(  1+ \frac{ |z_k-z'|}{\delta_D(z_k)}\right) + C
\le \frac12 \log \frac{ 1}{\delta_D(z_k)} + C,
\]
where $C$ depends on the choices made for $\tilde V_z, z'$. An analogous inequality
holds for $w_k$ and $w'$. Putting the inequalities together, we obtain
\[
k_D(z_k,z') + k_D(w_k,w') \le k_D(z_k,w_k) +C.
\]
Then \eqref{tbeg} implies
\[
T_D(z_k,w_k) \le   T_D(z',w') + C + k_D(z_k,w_k) + 2 C \omega_f (|z_\infty- w_\infty|),
\]
and since the first two terms are bounded, thus
 negligible in front of $k \omega_f (|z_k-w_k|^{1/m})$,
this contradicts the assumption.

{\it Case 2.3:} $z_\infty \in D, w_\infty \in \partial D$.

This time we only choose a neighborhood $\tilde V_w$ and a point $w' \in N_{w_\infty} \cap \partial \tilde V_w$.
The reasoning is analogous to the previous case, but simpler.
\end{proof*}

\subsection{Proof of Theorem \ref{finsloc}, \eqref{distlowest}.}
\ {}

To prove \eqref{distlowest} in Theorem \ref{finsloc},
we will switch the respective roles of $k_D$ and $T_D$.
We will need to establish some facts about the behavior of $T_D$, which is not as regular and well known as
the Kobayashi pseudo-distance.

It will be convenient to use some of the results proved above.  To this end,
given $t_D$ satisfying the hypothesis \eqref{metriclowest}, define
\begin{equation}
\label{deftz}
t_D^0(z;X):= \min \left( t_D(z;X) ; \left( 1+ f(\delta_D(z)) \right)^{-1} \kappa_D (z;X) \right).
\end{equation}
For $z\in U$, $t_D^0 (z;X) \le \kappa_D(z;X)$, and so $t_D^0$ satisfies both \eqref{metricupest} and \eqref{metriclowest}.
Obviously, $T_D^0 \le T_D$, so proving \eqref{distlowest} for $T_D^0$ will imply the same conclusion
for $T_D$.
So henceforth we assume that $t_D(z;X) = t_D^0(z;X)$, which implies in particular $t_D\le \kappa_D$
and $T_D \le k_D$.

We start with a rough intermediate estimate.
\begin{lemma}
\label{cstdiff}
Let $U$ be a $\C$-convexifiable neighborhood of $p$ such that every point of $U\cap \partial D$ is of type at most $m$ and that $D_U$ is Dini-smooth.
Let $z,w\in D_{U}$ be such that they can be joined by a $\mathcal C^1$-smooth $\eps$-extremal curve $\sigma$ for $T_D$ which remains inside $D_{U}$.
Then
\begin{equation}
\label{revcst}
k_D (z,w) \le  T_D (z,w) +C.
\end{equation}
\end{lemma}

\begin{proof}
We can repeat the proof of Lemma \ref{locprelim} with $t_D, T_D$ playing the respective roles of $\kappa_D, k_D$.
Reparametrize $\sigma$ by $t_D$-length.
\begin{multline}
\label{integbd2}
k_D(z,w) \le \int_a^b \kappa_D ( \sigma(t);\sigma'(t)) dt
\\
\le
\int_a^b \left( 1+ f\left(\delta_D(\sigma(t))\right) \right)  dt
\le T_D(z,w) + \eps +  \int_a^b  f\left(\delta_D(\sigma(t))\right)   dt.
\end{multline}

\begin{lemma}
\label{expdect}
There exists $s_0 \in [a';b']$ and $C_1, C_2>0$, independent of $a',b'$ and $\eps$ provided it is
small enough, such that for any $t\in [a';b']$
with $|t-s_0| \ge C_2$,
\[
\delta_D(\tilde \sigma(t)) \le C_1 |z-w|^{1/m} e^{-|t-s_0|}.
\]
\end{lemma}
\begin{proof}
Indeed,
\[
t_D (\tilde \sigma (s), \tilde \sigma (t) ) \le  k_D (\tilde \sigma (s), \tilde \sigma (t) ) \le k_{D_U} (\tilde \sigma (s),\tilde \sigma (t) )
 \]
and then we can use the rest of \eqref{logup} and the estimates following it to
 prove the Lemma in the same way as before.
\end{proof}
Lemma \ref{expdect} implies the analogue of \eqref{TkCint},
\begin{equation}
\label{kTCint}
k_D(z,w) \le T_D(z,w) + \eps+ 2 \min(C_2, b-a) f(\delta_D(\sigma (s_0))) + 2 \int_0^\infty f\left(C_1 |z-w|^{1/m} e^{-t}\right) dt
\end{equation}
setting $x:= C_1 |z-w|^{1/m} e^{-t}$, the last integral becomes
$\omega_f \left( C_1 |z-w|^{1/m}\right) \le C$,
using the fact that $D_{U_0}$ is bounded.
\end{proof}

\begin{lemma}
\label{geodcontrol}
Let $t_D$ be the Finsler pseudometric defined in \eqref{deftz}, and $p$ and $U$ be as in Lemma \ref{cstdiff}.
Let $V_0 \subset  \subset U$ be a neighborhood of $p$. Then
there exists $V_1 \subset \subset V_0$, another neighborhood of $p$, such that if $z,w \in V_1$,
and $\eps>0$ small enough, then any $\eps$-extremal curve $\sigma : [a;b] \rightarrow D$
 for $T_D$ joining $z$ to $w$ has to verify $\sigma([a;b])\subset V_0$.
\end{lemma}

\begin{proof}
Suppose to get a contradiction that there exists $s_0 \in [a;b]$ such that $\sigma(s_0)\notin V_0$.
Let $s_1:=\sup \{s: \sigma([a;s]) \subset V_0\}$,
$s_2:=\inf \{s: \sigma([s;b]) \subset V_0\}$, then $a<s_1<s_0<s_2<b$ and if we let $z':=\sigma(s_1)$,
$w':= \sigma(s_2)$, we have $z',w' \in \partial V_0$. We can apply \eqref{revcst} to get
\[
k_D (z,z') -C \le  T_D (z,z'), \quad k_D (w',w) -C \le  T_D (w',w),
\]
so that, by  \eqref{partlength}  and by \eqref{NAupest},
\begin{multline*}
k_D (z,z')+k_D (w',w) -2C \le T_D (z,z')+T_D (w',w)\le T_D (z,w)+3\eps  \\
\le k_D (z,w)+3\eps  \le  \log\left( 1+ \frac{C |z-w|}{\delta_D(w)^{1/2}\delta_D(z)^{1/2}} \right)+3\eps.
\end{multline*}
On the other hand, since a neighborhood of $p$ is $\C$-convexifiable, we have by \eqref{pisalow}
\begin{multline*}
k_D (z,z')+k_D (w',w) \ge \log   \frac{1}{\delta_D(w)^{1/2}} +  \log   \frac{1}{\delta_D(z)^{1/2}}
+m \log \mbox{dist}(\partial V_0, V_1) - C_2 \\
\ge \log   \frac{1}{\delta_D(w)^{1/2}\delta_D(z)^{1/2}}-C_3,
\end{multline*}
which contradicts the previous inequality when $V_1$, and therefore $|z-w|$, $\delta_D(z)$ and $\delta_D(w)$,
are small enough.
\end{proof}

Now we can give more precise bounds for $T_D$.
\begin{lemma}
\label{tbounds}
Let $t_D$ be the Finsler pseudometric defined in \eqref{deftz}, and $p\in \partial D$ as in Lemma  \ref{cstdiff}.
Then there exist a neighborhood $V$ of $p$ and constants $C', C''$ such that
for any $z,w \in V$,
\begin{equation}
\label{tdest}
m \log\left(  1+ \frac{C' |z-w|}{\delta_D(z)^{1/m}}\right)
\left(  1+ \frac{C' |z-w|}{\delta_D(w)^{1/m}}\right)
\le T_D(z,w) \le
 \log\left( 1+ \frac{C'' |z-w|}{\delta_D(z)^{1/2}\delta_D(w)^{1/2}} \right) .
\end{equation}
\end{lemma}

\begin{proof}
Since $T_D(z,w) \le k_D(z,w)$,  the right hand estimate follows by \eqref{NAupest}.

By \cite[Theorem 1.6]{NTh},
\[
k_D(z,w) \ge m \log\left(  1+ \frac{C_0 |z-w|}{\delta_D(z)^{1/m}}\right)
\left(  1+ \frac{C_0 |z-w|}{\delta_D(w)^{1/m}}\right) =: \log (1+A).
\]

For $V$ small enough, by Lemma \ref{geodcontrol},
any $\eps$-extremal curve $\sigma$ for $T_D$ will remain within $D\cap U$, with $U$ as in the hypothesis
\eqref{metriclowest} of Theorem \ref{finsloc};
so $\kappa_D(\sigma(t),X) \le C t_D(\sigma(t),X)$ for all $t$, and this proves $k_D(z,w)\le CT_D(z,w)$
whenever $z,w \in V$.  Along with Lemma \ref{cstdiff}, this shows that
\[
T_D(z,w) \ge \max(k_D(z,w) -C, \frac1C k_D(z,w)) \ge \max\left( \log(1+A)-C, \frac{\log(1+A)}C\right).
\]
By splitting into the two cases $A\le e^{2C}$ and $A\ge e^{2C}$, one sees that, taking $C_1:=e^{2C}$,
$T_D(z,w) \ge \log (1+\frac{A}{C_1})$, therefore \eqref{tdest} holds with $C'= C_0/C_1$.
\end{proof}

We resume the proof of Theorem \ref{finsloc}, \eqref{distlowest}.

By Lemma \ref{tbounds},  $T_D$ satisfies the  estimates \eqref{tdest},
which are the analogues for $T_D$ of \eqref{NAupest} (up to a constant)  and \eqref{pisalow} for $k_D$,
so that the analogue of Proposition
\ref{locgeod} holds for $\eps$-extremal curves for $T_D$. This sharpens
the control of $\eps$-extremal curves for $T_D$ already obtained in Lemma \ref{geodcontrol}.

We now prove the analogue of Lemma \ref{locprelim}:

\begin{lemma}
\label{locprelim2}
We can choose $U_0$ such that for any $p'\in U_0 \cap \partial D$,
there exists a small neighborhood $V'$ of $p'$ such that
 \eqref{distlowest} in Theorem \ref{finsloc} holds
when $z,w \in V'$.
\end{lemma}

\begin{proof}
We choose $U_0$  a neighborhood of $p$ satisfying all the above results. We want
$k_D (z;w) \le T_D (z;w) + C  \omega_f( |z-w|^{1/m}),$ for $z,w \in V'$.


Pick $z,w \in V'$ small enough so that any  (absolutely continuous) $\eps$-extremal curve for $T_D$ remains
within $U$, and that Lemmas \ref{geodcontrol} and \ref{tbounds} apply.  The analogue of Proposition
\ref{locgeod} implies that $\delta_D(\sigma (s_0)) \le \delta_D(z)+C_5|z-w|^{1/m}$, and that we can finish
the majorization of the last integral in \eqref{kTCint} as in the end of the proof of Lemma \ref{locprelim}.
\end{proof}

To have the result as stated in the theorem, we can again follow the same reasoning as for \eqref{distupest},
exchanging the roles of $t_D$ and $\kappa_D$.  As before, we proceed by contradiction and reduce the situation to the local situation, considering convergent sequences $(z_k)$ and $(w_k)$ such that
$k_D (z_k;w_k) \ge T_D (z_k;w_k) + k  \omega_f( |z_k-w_k|^{1/m})$.

If $z_\infty=w_\infty \in D$, it is easy to find a contradiction from the fact that $ T_D (z_k;w_k) \le C k_D (z_k;w_k) \le C'|z_k-w_k|$ for $k$ large enough. If $z_\infty=w_\infty \in \overline{V} \cap \partial D$,
with $ V \subset \subset U \subset U_0$,
by Lemma \ref{locprelim2} there is a neighborhood
$V'$ of $z_\infty$ such that Theorem \ref{finsloc} \eqref{distlowest} holds on it, again a contradiction.

The most delicate case is that when
$z_k \to z_\infty \in \partial D$, $w_k \to w_\infty \in \partial D$, and $z_\infty \neq w_\infty$.
Choose neighborhoods $V'_z$, $V'_w$ as in Lemma \ref{locprelim2},
$k$ large enough so $z_k \in  V'_z$,   $w_k \in V'_w$; and choose $\sigma$ an
 $\eps$-extremal curve for $T_D$ joining  $z_k$ to $w_k$.
 Let $z'$ be the first exit point from $V'_z$, precisely
$z':= \sigma(t_1)$, with $t_1:= \inf\{ t: \sigma(t)\notin  V'_z\}$, so $z'\in\partial  V'_z$, and
$w'$ the last entry point into $ V'_w$. Then the $\eps$-extremal  property implies
\[
T_D (z_k,w_k) \ge T_D (z_k,z') + T_D (w',w_k) -2\eps,
\]
and the local property we just proved implies
\[
T_D (z_k,z') + T_D (w',w_k) \ge k_D (z_k,z') + k_D (w',w_k) -C \omega_f(C),
\]
where the upper bound in the integral depends on the diameters of $ V'_z$ and $ V'_w$.
Again using \eqref{pisalow}, and the fact that $|z_k-z'|$ and $|w_k-w'|$ are bounded from below,
\[
k_D (z_k,z') + k_D (w',w_k) \ge \frac12 \log \frac1{\delta_D(z_k)} + \frac12 \log \frac1{\delta_D(w_k)} -C,
\]
where once again the constant depends on  the choices made for $V'_z$, $V'_w$. Now
again by \cite{NPT1},
\begin{equation*}
k_D(z_k,w_k) 
\le \frac12 \log \frac{ 1}{\delta_D(z_k)} + \frac12 \log \frac{ 1}{\delta_D(w_k)} +C,
\end{equation*}
since $|z_k-w_k|\to |z_\infty-w_\infty|>0$. This implies that $ k_D (z_k,w_k)-T_D (z_k,w_k) \le C$ and
concludes the proof by contradiction in this case.

The last case is analogous but simpler.

\subsection{Proof of Theorem \ref{finsloc}, \eqref{uprat}  and \eqref{lowrat}.}

\begin{prop}
\label{localrat}
Under the hypothesis \eqref{metricupest} of Theorem \ref{finsloc}, for any $q\in U\cap \partial D$, there exists
a neighborhood $V_q\subset U$ of $q$ and $c>0$ such that for any $z,w\in V_q$,
\begin{equation}
\label{locuprat}
 \frac{T_D(z,w)}{k_D(z,w)} \le 1+ f\left( \delta_D(z)+ c |z-w|^{1/m}\right) ;
\end{equation}
and under the hypothesis \eqref{metriclowest},
for any $q\in U\cap \partial D$, there exists
a neighborhood $V_q\subset U$ of $q$ and $c>0$ such that for any $z,w\in V_q$,
\begin{equation}
\label{loclowrat}
 \frac{k_D(z,w)}{T_D(z,w)} \le 1+ f\left( \delta_D(z)+ c |z-w|^{1/m}\right).
\end{equation}
\end{prop}
\begin{proof}
To prove \eqref{locuprat}, choose $V_q$ small enough so that Proposition \ref{locgeod} applies.
Then for $\eps$ small enough and an $\eps$-extremal curve $\sigma$ for $k_D$ joining $z$ to $w$, we have for all $t$,
\[
\delta_D(\sigma(t)) \le \delta_D(z)+ C |z-w|^{1/m},
\]
therefore
$t_D(\sigma(t);\sigma'(t)) \le \left(1+ f\left( \delta_D(z)+ C |z-w|^{1/m} \right) \right) \kappa_D(\sigma(t);\sigma'(t))$.
Thus
\begin{multline*}
T_D(z,w) \le \left(1+ f\left( \delta_D(z)+ C |z-w|^{1/m} \right) \right) \int_a^b \kappa_D(\sigma(t);\sigma'(t)) dt
\\
\le \left(1+ f\left( \delta_D(z)+ C |z-w|^{1/m} \right) \right) \left( k_D(z,w) + \eps \right),
\end{multline*}
and we have the result letting $\eps$ go to $0$.

To prove \eqref{loclowrat}, we switch the roles of $T_D$ and $k_D$.
As explained before Lemma \ref{locprelim2}, the analogue of Proposition \ref{locgeod}
gives us
the same control over the location of the $\eps$-extremals for $T_D$ as we had previously from for $k_D$.
\end{proof}

We obtain the statements about any $V\subset \subset U$ from the Proposition as we had obtained the
corresponding statement for \eqref{distupest}: construct sequences of points $z_k, w_k$
such that
\[
\frac{T_D(z_k,w_k)}{k_D(z_k,w_k)} \ge 1+ f\left(  \delta_D(z)+ k |z-w|^{1/m}\right),
\]
and pass to convergent subsequences.
If $(z_k)$ and $(w_k)$ converge to distinct points inside $D$, it is easy to find a contradiction,
using the fact that $\lim_{x\to\infty}f(x)=\infty$; if one of the limits is in $\partial D$,
by \cite[Theorem 2.3]{FR}, $k_D(z_k,w_k)\to\infty$, and the fact that $T_D \le k_D +C$ is violated.  If the limits
coincide inside $D$, the infinitesimal hypothesis \eqref{metricupest} is violated near the limit
point, and if the common limit lies on $\partial D$, then the Proposition above is contradicted.
The same reasoning goes through for the reverse inequality.

\section{Proofs of Theorems \ref{berg}, \ref{lemp}, and \ref{lempcara}}
\label{lempert}

\subsection{Proof of Theorem \ref{berg}.}

We will follow the pattern of the proof of Theorem \ref{kobloc}:
we obtain  inequalities of comparison of infinitesimal metrics analogous to \eqref{kobupest}, which hold
locally; then using Lemma \ref{connhd}, we obtain an open set on which we can
apply Theorem \ref{finsloc} for the relevant metric.

A central juncture of the proof will be to use properties of the squeezing function on an appropriate
strictly pseudoconvex subdomain of $D$.  We choose $U_0$ a neighborhood of $p$ such that $D_{U_0}$
is strictly pseudoconvex, with $\mathcal C^{k,\eps}$-smooth boundary (the regularity given by the hypotheses).
We assume $V$, $U$ are open neighborhoods of $p$ as in the statement of the theorem.

\begin{lemma}
\label{manym}
For any $\nu, \xi$ chosen among the six pseudometrics
$ \kappa_D, \kappa_{D_{U_0}}, \kappa_{D_U}, \tilde \beta_{D}, \tilde \beta_{D_{U_0}}, \tilde \beta_{D_U}$,
and any $q\in (\partial D)\cap U$, there exists an open neighborhood of $q$, $V_q\subset U$ such
that $D_{V_q}$ is connected and for any $z\in D_{V_q}$, $X\in \C^n$,
\[
\left| \frac{\nu(z;X)}{\xi(z;X)} - 1\right| \le C \delta_D (z)^{(k-2+\eps)/2}.
\]
\end{lemma}

\begin{proof}
We start with $(\nu,\xi)= (\kappa_{D_{U_0}}, \tilde \beta_{D_{U_0}})$.

Recall that the squeezing function for a domain $G$, $z\mapsto \sigma_G(z)$, is a holomorphically invariant
function with values between $0$ and $1$ that measures
how much  $G$ looks like a ball when seen from $z$.
For a more precise definition and properties, see \cite{FW}, \cite{NTr}
and references therein.  The salient fact for us is (see e.g. \cite[p. 1361]{NTr})
\begin{equation}
\label{sqzratio}
\sigma_G(z)^{n+1} \le \frac{\tilde \beta_G(z;X)}{\kappa_G(z;X)} \le \sigma_G(z)^{-n-1} .
\end{equation}

\cite[Theorem 1]{NTr}, itself springing from \cite{FW},
implies that if $G$ is strictly pseudoconvex with boundary of class $\mathcal C^{k,\eps}$, $k\in \{2,3\}$,
$0<\eps\le 1$, then there exists $C>0$ such that
\begin{equation}
\label{sqzest}
1\ge \sigma_G(z) \ge 1 - C \delta_G(z)^{(k-2+\eps)/2}.
\end{equation}
We apply this to $ D_{U_0}$ and get
\begin{equation}
\label{sqzratio2}
1 - C \delta_{D_{U_0}}(z)^{(k-2+\eps)/2} \le \frac{\tilde \beta_{D_{U_0}}(z;X)}{\kappa_{D_{U_0}}(z;X)} \le 1 + C \delta_{D_{U_0}}(z)^{(k-2+\eps)/2}.
\end{equation}
Choosing $V_q$ small enough, we have $\delta_{D_{U_0}}(z) \le \delta_{D}(z)$ for $z\in D_{V_q}$. So the Lemma
is proved for that pair of metrics.

It is now enough to compare the various $\kappa_G$ and the various $\tilde \beta_G$
between themselves.

If $(\nu,\xi)= (\kappa_{D_{U_0}}, \kappa_{D})$ or $(\kappa_{D_{U}}, \kappa_{D_{U_0}})$, the
Lemma holds because of \eqref{kobupest}, since $1> (k-2+\eps)/2$, and thus it holds for
$(\nu,\xi)= (\kappa_{D_{U}}, \kappa_{D})$ too.

Since $q$ is a   strictly pseudoconvex point of $\partial D$, it follows from \cite[Lemma 3]{N2} that for any
neighborhood $U$ of $q$ such that $D\cap U$ is connected, there exists a neighborhood $W$ of $q$ such
that for all $z\in W$,
\begin{equation}
\label{betaloc}
\left| \frac{\tilde \beta_{D\cap U} (z;X)}{\tilde \beta_{D} (z;X)} -1 \right|
\le
C \delta_D (z) \log\frac1{\delta_D (z)}.
\end{equation}
So the Lemma holds for $\nu,\xi \in \left\{ \tilde \beta_{D}, \tilde \beta_{D_{U_0}}, \tilde \beta_{D_{U}}\right\}$.
\end{proof}

To finish the proof of Theorem \ref{berg}, we use Lemma \ref{connhd}
to obtain an open set $W_0$ such that $ V\subset \subset W_0 \subset \subset U$ and $D\cap W_0$
is connected. We take a finite covering of $\overline W_0 \cap \overline D$ by $V_{q_j}$, $q_j\in \overline W_0$,
 $1\le j \le N$,
where the $V_{q_j}$ are chosen as in Lemma \ref{manym}, and let $W_1:= \bigcup_{j=1}^N V_{q_j} \subset \subset U$.
Since $D\cap W_1  = (D\cap W_0) \cup \bigcup_{j=1}^N (D\cap V_{p_j})$, it is connected as well.
The hypotheses \eqref{metricupest} and \eqref{metriclowest} of
Theorem \ref{finsloc} apply with $f(s)=cs^{(k-2+\eps)/2}$ on $W_1$.

We can apply
Theorem \ref{finsloc} taking $t_D$ to be any metric in
$\left\{ \kappa_{D_{U}}, \kappa_{D_{U_0}}, \tilde \beta_{D}, \tilde \beta_{D_{U_0}}, \tilde \beta_{D_{U}}\right\}$,
with $U:=W_1$ and $f$ as above.
Then $\omega_f(s)=c's^{(k-2+\eps)/2}$.  In particular, when either $\nu=t_D $ or
$\xi=\kappa_D$, then \eqref{distupest} and \eqref{distlowest} imply the desired inequalities
about differences of distances in \eqref{kobberg} when, say, $v=k_D$. But then, since we can compare
any two distances with $k_D$, we can compare them between themselves, too.

Lastly, we need to obtain the statement in \eqref{kobberg} about ratios of distances.
This follows from applying the conclusions \eqref{uprat} and \eqref{lowrat} of Theorem \ref{finsloc}.
\vskip.3cm

\noindent
{\it Proof of Corollary \ref{cor14}.}
As a consequence of \cite[Proposition 2]{N2}, $| k_D(z,w)-b_D(z,w)| \le C$
on $D$.

Let us consider first the cases $k=2$ or $\eps<1$. If we had sequences of points
$(z_k)_k, (w_k)_k$ in $D$
so that
\begin{equation}
\label{contracoro}
k |z_k-w_k|^{(k-2+\eps)/4} \le | k_D(z,w)-b_D(z,w)|,
\end{equation}
then, passing to convergent subsequences,
 $(z_k)_k$ and $(w_k)_k$ would have to tend to the same point $z_\infty$. If $z_\infty \in D$,
this would contradict the fact that on a compactum $K$, $k_D(z,w), b_D(z,w) \le C_K \|z-w\|$.
So $z_\infty \in \partial D$, but then this contradicts Theorem \ref{finsloc} for the point $z_\infty$.

Similarly, suppose we had sequences $(z_k)_k, (w_k)_k$ in $D$
so that
\begin{equation}
\label{contracororat}
\left| \frac{k_D(z,w)}{b_D(z,w)} -1\right| \ge k \left( \delta_D(z)+|z-w|^{1/2}\right)^{(k-2+\eps)/2},
\end{equation}
Passing to convergent subsequences, we may assume $z_k \to z_\infty \in \overline D$ and
$w_k \to w_\infty \in \overline D$. If $z_\infty , w_\infty \in D$, then $z_\infty = w_\infty$.
But then $k_D(z,w), b_D(z,w) \asymp  \|z-w\|$ (on a compactum) and $\delta_D(z_\infty)>0$
lead to a contradiction.

If $\lim_k k_D(z_k,w_k)=\infty$, then $| k_D(z_k,w_k)-b_D(z_k,w_k)| \le C$ implies that
$\frac{k_D(z_k,w_k)}{b_D(z_k,w_k)}\to 1$.  But if $z_\infty \in D$ and $w_\infty \in \partial D$,
this will be the case, and so lead to another contradiction.

We may now assume $z_\infty, w_\infty \in \partial D$.  If $z_\infty \neq w_\infty$, then
again \cite[Theorem 2.3]{FR} implies that $\lim_k k_D(z_k,w_k)=\infty$, so we have another contradiction.
Finally $z_\infty= w_\infty \in \partial D$, and again this  contradicts Theorem \ref{finsloc}.

In the case $k=3$, $\eps=1$, we exploit the fact that since $D$ is strictly pseudoconvex, we can take $D_{U_0}=D$,
and we do not need to use the estimate \eqref{betaloc}. Furthermore Lemma \ref{manym} holds with $D_U=D$.
So we can apply directly Theorem \ref{finsloc} with $f(s)=cs$ and get the desired result.

\subsection{Proof of Theorem \ref{lemp}.}

\begin{proof}

It follows from \cite[Theorem 1]{NPT1} that if $D$ is bounded and
$\partial D$ is $\mathcal C^{1,\eps}$-smooth for some $\eps>0$, then there exists $C>0$ such that
\begin{equation}
\label{lup}
l_D(z,w) \le \frac12 \log \frac1{\delta_D(z)}+\frac12 \log \frac1{\delta_D(w)}+C.
\end{equation}
On the other hand, \cite[Theorem 1.6]{NTh} states that under the hypotheses of Theorem \ref{lemp},
\begin{equation}
\label{klow}
k_D(z,w) \ge \frac{m}2 \log \left( 1+C \frac{|w-z|}{\delta_D(z)^{1/m}} \right)
+ \frac{m}2 \log \left( 1+C \frac{|w-z|}{\delta_D(w)^{1/m}} \right).
\end{equation}
If \eqref{lkvd} failed, there would be two sequences $(z_k)_k, (w_k)_k$ in $D$
so that
\begin{equation}
\label{contrakd}
k |z_k-w_k|^{1/m} \le l_D(z_k,w_k) - k_D(z_k,w_k).
\end{equation}
Passing to subsequences, we may assume $z_k\to z_\infty \in \overline D$, $w_k\to w_\infty \in \overline D$.

If $z_\infty \neq  w_\infty$, then there is some $\eta>0$ such that for $k\ge N$,
$|z_k-w_k|\ge \eta$.
But \eqref{lup} and \eqref{klow} taken together imply
\[
l_D(z_k,w_k) - k_D(z_k,w_k) \le -2 \log (C|z_k-w_k|^{1/m}) \le -2 \log (C\eta),
\]
and this contradicts \eqref{contrakd}.

So we are reduced to the case $z_\infty =  w_\infty$. If $z_\infty \in D$, then
for $k$ large enough, $l_D(z_k,w_k) \le C |z_k-w_k|$, and we have a contradiction again.
We then assume $z_\infty \in \partial D$. Let $U_0$ be a neighborhood of $z_\infty$
where Theorem \ref{kobloc} applies. By Lemma \ref{locCcvx}, we can take
 neighborhoods $V \subset \subset U \subset U_0$ of $z_\infty$
such that $D_{U}$ is $\C$-convexifiable and $\mathcal C^m$ smooth.

Since it is at least $\mathcal C^2$-smooth,
 we can apply Jacquet's extension of Lempert's theorem \cite{Jac}
 to the $\C$-convexifiable domain $D_{U}$, so that
$l_{D_{U}}=k_{D_{U}}$, and Theorem \ref{kobloc} \eqref{kobloceq} yields,
for $k$ large enough so that $z_k, w_k \in V$,
\[
l_D(z,w)  \le l_{D_{U}}(z_k,w_k) = k_{D_{U}}(z_k,w_k) \le k_D(z_k,w_k) + C |z_k-w_k|^{1/m},
\]
which contradicts \eqref{contrakd}.

To prove \eqref{lkrat}, we proceed in a similar way: assume there are
 sequences $(z_k)_k, (w_k)_k$ in $D$ such that
\begin{equation}
\label{contrakr}
 1+ k \left( \delta_D(z_k)+|z_k-w_k|^{1/m}\right) \le \frac{l_D(z_k,w_k)}{k_D(z_k,w_k)} ,
\end{equation}
and pass to convergent subsequences.  The case $z_\infty \neq  w_\infty$ can be ruled out
because then, since $D$ is bounded, $k_D(z_k,w_k) \ge \eta >0$ for $k$ large enough,
and then \eqref{lup} and \eqref{klow} taken together again imply that the right hand
side of \eqref{contrakr} is bounded.

When $z_\infty =  w_\infty$, again one sees that $z_\infty \in D$ leads to a contradiction
because $l_D(z_k,w_k)$ and $k_D(z_k,w_k)$ are both comparable to $|z_k-w_k|$
 and $\delta_D(z_k)\ge \eta>0$, for $k$ large enough.  Finally, if $z_\infty \in \partial D$,
we pick neighborhoods $V \subset \subset U \subset U_0$
as above, but use the conclusion \eqref{kobrat} from Theorem \ref{kobloc} to get a final contradiction.

\end{proof}

\subsection{Proof of Theorem \ref{lempcara}.}

\begin{proof}
We follow the pattern of the proof of \cite[Proposition 1]{N2}. Assume the theorem fails.  Then we
may find sequences of points $(z_k)_k, (w_k)_k$, converging respectively to $p, q \in \overline D$, such that
$(l_D(z_k,w_k) - c_D(z_k,w_k))  |z_k-w_k|^{-1/2} \to \infty$.

For a strictly pseudoconvex domain $D$, by \cite[Proposition 1]{N2},
\begin{equation}
\label{caranik}
c_D(z,w) \le k_D(z,w) \le l_D(z,w) \le c_D(z,w) +C.
\end{equation}
So we must have $p= q$. If $p\in D$,
then for $k$ large enough, $0\le c_D(z_k,w_k) \le l_D(z_k,w_k) = O(|z_k-w_k|)$, which easily implies the result.

Assume now $p=q\in \partial D$.
Since $p$ is a strictly pseudoconvex boundary point, we may apply Fornaess' embedding theorem \cite[Proposition 1]{For}
 to obtain a holomorphic map $\Phi$ from a neighborhood of $\overline D$ to $\mathbb C^n$
  and a strictly convex domain $G \supset \Phi(D)$ such that, near $p$, $\Phi$ is biholomorphic and $\partial G = \partial D$.
Now we may choose a neighborhood $U$ of $p$ such that $\Phi$ is biholomorphic on
a neighborhood of $\overline U$ and $G' := \Phi(D_U)= G_V$ is a strictly convex domain
and $V$ is a neighborhood of $\Phi(p)$. For
$k$ large enough so that $z_k, w_k \in D_U$,
 let $z_k' =\Phi(z_k)$ and $w_k' =\Phi(w_k)$. Note that, since $\Phi$ is biholomorphic on
a neighborhood of $\overline U$, there is a uniform constant $C>0$ such that $C^{-1} |z_k-w_k| \le |z'_k-w'_k| \le C |z_k-w_k|$.
Then
\begin{multline}
\label{chainbed}
l_D(z_k,w_k) - c_D(z_k,w_k) \le
l_{D_U}(z_k,w_k) - c_D(z_k,w_k) = l_{G'}(z'_k,w'_k) - c_{\Phi(D)}(z'_k,w'_k) \\
\le l_{G'}(z'_k,w'_k) - c_{G}(z'_k,w'_k) =  k_{G'}(z'_k,w'_k) - k_{G}(z'_k,w'_k) ,
\end{multline}
by applying Lempert's theorem to $G$ and $G'$. So we would get
\newline
$(k_{G'}(z'_k,w'_k) - k_{G}(z'_k,w'_k)) |z'_k-w'_k|^{-1/2}\to \infty$,
a contradiction to \eqref{kobloceq} in the case $m=2$.

We adapt the previous proof to obtain \eqref{lcrat}.  Assume now we have sequences of points $(z_k)_k, (w_k)_k$, converging respectively to $p, q \in \overline D$, such that
\[
\frac{ \frac{l_D(z_k,w_k)}{c_D(z_k,w_k)}-1 }{\delta_D(z_k)+ |z_k-w_k|^{1/2} }\to \infty.
\]
If $p\neq q$, the denominator remains bounded below, and the numerator is bounded, by \eqref{caranik}, so
we have a contradiction. If $p=q \in D$, then because there are constants so that $c|z-w|\le c_D(z,w) \le l_D(z,w) \le C|z-w|$,
we have a contradiction again. So the only remaining case is when $z_k, w_k \to p\in \partial D$, and we follow
the construction above, replacing \eqref{chainbed} by:
\begin{multline*}
\frac{l_D(z_k,w_k)}{c_D(z_k,w_k) } \le
\frac{l_{D_U}(z_k,w_k)}{c_D(z_k,w_k)} = \frac{l_{G'}(z'_k,w'_k)}{c_{\Phi(D)}(z'_k,w'_k)}
\le \frac{l_{G'}(z'_k,w'_k)}{c_{G}(z'_k,w'_k)}
=  \frac{k_{G'}(z'_k,w'_k)}{k_{G}(z'_k,w'_k)}
 \\ \le 1+C' ( \delta_G(z'_k)+|z'_k-w'_k|^{1/2})
\le 1+C ( \delta_D(z_k)+|z_k-w_k|^{1/2}),
\end{multline*}
by Theorem \ref{kobloc} \eqref{kobrat}, a contradiction again.
\end{proof}
Observe that this proof circumvents the use of the result of Balogh and Bonk \cite{BB},
since the proof of \cite[Proposition 1]{N2} only used \cite{BB} to deal with the case when $z$ and $w$
were close to the same boundary point $p$, and this case is now settled by Theorem \ref{kobloc}.

\section{The planar case}
\label{plane}

Let $D\subset \mathbb C$ be a planar domain; 
we write $D_r=D\cap(r\D)$ $(r>0$).
Denote by $\Pi$ the upper half plane. We write $m_D=\tanh k_D$, in particular $m_\Pi(z,w)=\left| \frac{z-w}{z-\bar w}\right|$.

Recall that $k_D=l_D$ (see e.g. \cite[Remark 3.3.8.(e)]{JP}).

\begin{proposition}\label{uhp}
For $z,w \in \Pi_1=\Pi\cap \D$, $k_{\Pi_1}(z,w)-k_\Pi(z,w) = T_1(z,w)+T_2(z,w)$, where
\begin{equation}\label{t1}
T_1(z,w)= \log \left( 1+ |z-w| \frac{\Im z \Im w}{|z- w|+|z-\bar w|} \frac{4}{(|1-z w| + |1-z \bar w|)|1-z \bar w|} \right),
\end{equation}
\begin{equation}\label{t2}
\mbox{and } T_2(z,w)= -\frac12 \log \left(1- \frac{|z-w|^2}{|1-z\bar w|^2} \right).
\end{equation}
When $z,w \to 0$
\begin{equation}
\label{tla1}
T_1(z,w)= 2 |z-w| \frac{\Im z \Im w}{|z- w|+|z-\bar w| }  (1+ o (1)),
\end{equation}
\begin{equation}
\label{tla2}
T_2(z,w)= \frac12 |z-w|^2 (1+ o (1)).
\end{equation}

As a consequence of \eqref{t1} and \eqref{t2}, for any $\e>0$ there exist $r\in(0,1)$ such that
\begin{equation}
\label{mdl}
0\le k_{\Pi_1}(z,w)-k_\Pi(z,w)<(1+\e)|z-w|(\frac12 |z-w|+\min(\Im z, \Im w))
\end{equation}
when $z,w\in\Pi_r.$ When $\Re z=\Re w$, this  estimate is sharp:
\newline
$k_{\Pi_1}(z,w)-k_\Pi(z,w)\sim |z-w|(\frac12 |z-w|+\min(\Im z, \Im w))$.
\end{proposition}

\begin{proof} Since the map  $f(z):=\left(\frac{z+1}{z-1}\right)^2$ transforms
conformally $\Pi_1$ onto $\Pi,$ $m_{\Pi_1}(z,w)= m_\Pi(f(z),f(w))$. Notice also that
\begin{equation}
\label{1mm}
1-m_\Pi(z,w)^2 = \frac{4 \Im z \Im w}{|1-z\bar w|^2}.
\end{equation}
Inverting the hyperbolic
tangent function,
\[
k_D(z,w)= \log (1+m_D)- \frac12 \log (1-m_D(z,w)^2), \mbox{ thus }
\]
\[
k_{\Pi_1}(z,w)-k_\Pi(z,w) = \log \frac{1+m_{\Pi_1}(z,w))}{1+m_\Pi(z,w))}
- \frac12 \log \frac{1-m_{\Pi_1}(z,w)^2}{1-m_\Pi(z,w)^2}.
\]
Elementary computations and the fact that $\overline{f(z)}= f(\bar z)$ show that
$$
f(z)-f(w)= 4 \frac{(1-zw)(z-w)}{(1-z)^2(1-w)^2}, \mbox{ and }
\Im f(z)= 4 \frac{(1-|z|^2) \Im z}{|1-z|^4}.
$$
We deduce that
\begin{multline*}
m_{\Pi_1}(z,w) = \left| \frac{1-zw}{1-z\bar w}\right| m_{\Pi}(z,w)=: \mu_1 m_{\Pi}(z,w) ,
\mbox{ and } \\
1-m_{\Pi_1}(z,w)^2= \left( 1-m_\Pi(z,w)^2\right) \frac{(1-|z|^2) (1-|w|^2)}{|1-z\bar w|^2}.
\end{multline*}
We then write $\frac{1+\mu_1 m_\Pi(z,w)}{1+ m_\Pi(z,w)}= 1 + (\mu_1-1)\frac{ m_\Pi(z,w)}{1+ m_\Pi(z,w)}$,
and
$$
\mu_1-1 = \frac{\mu_1^2-1}{\mu_1+1}=  \frac{1}{\mu_1+1} \frac{4 \Im z \Im w}{|1-z\bar w|^2}.
$$
Elementary computations then lead to \eqref{t1} and \eqref{t2}.

Since $|1-z\bar w|\to 1$ as $z,w\to0$, the asymptotic expansions in \eqref{tla1}, \eqref{tla2}
follow readily.

To get \eqref{mdl}, notice that
\[
|z- w|+|z-\bar w|\ge |\Im z - \Im w|+|\Im z + \Im w|=2 \max (\Im z , \Im w).
\]
This becomes an equality when $\Re z=\Re w$.
\end{proof}

\noindent{\bf Remark.}  None of the two additive terms in the estimate \eqref{mdl} can be removed.
This can be seen by choosing points $z,w \in i\R_+$ such that either $|z-w|$ or $\min (\Im z , \Im w)$ is  dominating.

\begin{proof*}{\it Proof of Theorem \ref{dimone}.}
We may choose $U_0$ such that
$D_{U_0}$ is a bounded Dini-smooth domain.

Assume the proposition fails.
Then we find sequences $z_k\to p'\in\overline{D_V}$ and
$w_k\to p''\in\overline{D_V}$  such that
$$\frac{k_{D_U}(z_k,w_k)-c_D(z_k,w_k)}{|z_k-w_k|(|z_k-w_k|+\d_D(z_k)^{1/2}\d_D(w_k)^{1/2})}
\to\infty.$$

We distinguish the same three cases as in the proof of Theorem \ref{lempcara}.

When $p'\neq p'',$ we get a contradiction to \cite[Propositions 5 and 7, Corollary 6]{N1} which
imply that there exists a constant $c>0$ such that if $z$ is near $p'$ and $w$ is near $p'',$ then
$$2k_{D_U}(z,w)-c<-\log\d_D(z)-\log\d_D(w)<2c_D(z,w)+c.$$

If $p'=p''\in D_V,$ then we use the fact that  $k_{D_U}(z,w)\le c|z-w|$ for $z,w$ near $p.$

Let now $p'=p''\in\partial D.$ We shall use arguments from the proof of \cite[Proposition 6]{NTA}.
We may find a Dini-smooth Jordan curve $\zeta$ such that
$\zeta$ coincides with $\partial D$ near $p'$ and $D\subset\zeta_{\mbox{ext}}.$
Since $V\subset \subset U$ and $p'\in \bar V$, we can fix some $r_0>0$ depending only on $U$ and $V$ such that
$D(p',r_0)\subset U$ and $D(p',r_0)\cap \partial D = D(p',r_0)\cap \zeta$.

Let $a\in\zeta_{\mbox{int}},$
$\ds\varphi(z)=\frac{1}{z-a}$ and $G=\varphi(\zeta_{\mbox{ext}})\cup\{0\}.$
Let $\psi:G\to\D$ be a Riemann map. It extends to a $\mathcal C^1$-diffeomorphism
from $\overline G$ to $\overline \D$ (cf. \cite[Theorem 3.5]{Pom}).
Let $\theta$ map conformally $\D$ onto $\Pi,$ chosen so that $\theta \left( \psi \circ \varphi(p') \right)=0$.

Set $\eta=\theta\circ\psi\circ\varphi$; for some $r>0$ (depending on $\eta$ and $r_0$), $D_r \subset \eta (D_U)$.
Writing $\z=\eta(z)$, we  then have
$$c_D(z_k,w_k)\ge c_\Pi(\z_k,\w_k)=k_\Pi(\z_k,\w_k),\quad k_{D_U}(z_k,w_k)\le k_{\Pi_r}(\z_k,\w_k)$$
for $k$ large enough.

Since $c_\Pi=k_\Pi,$ $\ds\lim_{z\to a}\frac{\d_\Pi(\z)}{\d_D(z)}=|\eta'(a)|,$
and $\ds\frac{\z_k-\w_k}{z_k-w_k}\to\eta'(a),$ then using Proposition \ref{uhp}
$$
0\le k_{D_U}(z_k,w_k)-c_D(z,w)\le k_{\Pi_r}(\z_k,\w_k)-
k_{\Pi}(\z_k,\w_k)
$$
$$
\le C_r |\z_k-\w_k|(|\z_k-\w_k|+\d_D(\z_k)^{1/2}\d_D(\w_k)^{1/2})$$
$$
\le C|z_k-w_k|(|z_k-w_k|+\d_D(z_k)^{1/2}\d_D(w_k)^{1/2})
$$
for some $C>0$ and any $k$ large enough. We reach again a contradiction
which completes the proof.
\end{proof*}

\begin{corollary}\label{ck} Let $D$ be a bounded Dini-smooth planar domain $D.$
Then there exists a constant $C>0$ such that
$$
0\le k_D(z,w)-c_D(z,w)\le C|z-w|(|z-w|+ \d_D(z)^{1/2}\d_D(w)^{1/2}),\quad z,w\in D.
$$
\end{corollary}

\begin{proof} Again we follow the same proof pattern:

(i) By \cite[Proposition 8]{N1}, the difference $k_D-c_D$ is bounded on
$D\times D;$

(ii) For any $p\in D,$  there is a constant $C_p>0$ such that for $z,w$ near $p$,
$k_D(z,w)\le C_p|z-w|$;

(iii) Theorem \ref{dimone} applied to any boundary point and replacing
$k_{D_U}$ by $k_D.$
\end{proof}

{}

\end{document}